\documentclass[12pt]{amsart}
\usepackage{amssymb,amsfonts,amsmath,amsopn,amstext,amscd,latexsym, amsthm, enumerate,mathrsfs, mathabx}

\usepackage{color}
\usepackage{xcolor}
\usepackage[colorlinks=true]{hyperref}
\hypersetup{
     colorlinks   = false,
     citecolor    = blue
}

\usepackage[margin=30mm]{geometry}
\usepackage{bbm}
\usepackage[all,cmtip]{xy}

\usepackage{enumerate}
\usepackage[shortlabels]{enumitem}

\newtheorem{theorem}{Theorem}[section]
\newtheorem{thm}[theorem]{Theorem}
\newtheorem{lemma}[theorem]{Lemma}

\newtheorem{thmalph}{Theorem}

\newtheorem{lem}[theorem]{Lemma}

\newtheorem{proposition}[theorem]{Proposition}

\newtheorem{corollary}[theorem]{Corollary}

\theoremstyle{definition}

\theoremstyle{remark}

\DeclareMathOperator{\Real}{Re}

\DeclareMathOperator{\Aut}{Aut}

\DeclareMathOperator{\diag}{diag}
\DeclareMathOperator{\Irr}{Irr}

\DeclareMathOperator{\GL}{GL}
\DeclareMathOperator{\SL}{SL}
\DeclareMathOperator{\SU}{SU}

\newcommand{\bbF}{{\mathbb F}}

\newcommand{\bbR}{{\mathbb R}}

\newcommand{\out}{\mathrm{Out}}
\newcommand{\R}{{\mathbb R}}

\newcommand{\Ker}{{\mathrm {Ker}}}

\newcommand{\PSL}{{\mathrm {PSL}}}

\newcommand{\PSU}{{\mathrm {PSU}}}

\newcommand{\PSp}{{\mathrm {PSp}}}

\DeclareMathOperator{\Sol}{Sol}

\newcommand{\RR}{\mathbb{R}}
\newcommand{\Fit}{\mathbf{F}}
\newcommand{\OO}{\mathbf{O}}
\newcommand{\Centralizer}{\mathbf{C}}

\numberwithin{equation}{section}

\newcommand{\Out}{{\mathrm {Out}}}

\newcommand{\Elt}{\mathcal{E}}
\newcommand{\Layer}{\mathbf{E}}

\newcommand{\bg}[1]{\textbf{#1}}
\newcommand{\wt}[1]{\widetilde{#1}}

\newcommand{\bC}{{\mathbf{C}}}
\newcommand{\bF}{{\mathbf{F}}}

\newcommand{\bN}{{\mathbf{N}}}

\newcommand{\Al}{\textup{\textsf{A}}}
\newcommand{\Sy}{\textup{\textsf{S}}}

\newcommand{\irr}{\mathrm{Irr}}

\begin{document}

\title[Real-valued characters]{ON THE NUMBER OF IRREDUCIBLE REAL-VALUED
CHARACTERS OF A FINITE GROUP}

\author[N.\,N. Hung]{Nguyen Ngoc Hung}
\address{Department of Mathematics, The University of Akron, Akron, OH 44325,
USA} \email{hungnguyen@uakron.edu}

\author[A.\,A. Schaeffer Fry]{A. A. Schaeffer Fry}
\address{Department of Mathematical and Computer Sciences, Metropolitan State University of Denver,
Denver, CO 80217, USA} \email{aschaef6@msudenver.edu}

\author[H.\,P. Tong-Viet]{Hung P. Tong-Viet}
\address{Department of Mathematical Sciences, Binghamton University, Binghamton, NY 13902-6000, USA}
\email{tongviet@math.binghamton.edu}

\author[C.\,R. Vinroot]{C. Ryan Vinroot}
\address{Department of Mathematics, College of William and Mary,
Williamsburg, VA, 23187, USA} \email{vinroot@math.wm.edu}

\begin{abstract} We prove that there exists an integer-valued function $f$
on positive integers such that if a finite group $G$ has at most $k$
real-valued irreducible characters, then $|G/\Sol(G)|\leq f(k)$,
where $\Sol(G)$ denotes the largest solvable normal subgroup of $G$.  In the case
 $k=5$, we further classify $G/\Sol(G)$.
This partly answers a question of Iwasaki \cite{Iwa} on the relationship
between the structure of a finite group and its number of
real-valued irreducible characters.
\end{abstract}

\thanks{The paper was initiated while the authors were in residence at MSRI (Berkeley, CA) during the Spring 2018
semester (supported by the NSF under grant DMS-1440140). We thank
the Institute for the hospitality and support. The second author was supported in part by NSF grant DMS-1801156.  The fourth author was supported in part by a grant from the Simons Foundation, Award \#280496}


\date{May 23, 2019}

\maketitle

\section{Introduction}

Analyzing fields of character values is a difficult problem in the
representation theory of finite groups.  Real-valued characters and rational-valued characters have received
more attention than others.

It is well-known that a finite group $G$ has a unique
real/rational-valued irreducible character if and only if $G$ has odd
order. In \cite{Iwa}, Iwasaki proposed to study the relationship between
the structure of $G$ and the number of real-valued
irreducible characters of $G$, which we denote $k_\R(G)$. He showed that if $k_\R(G) = 2$, then
$G$ has a normal Sylow $2$-subgroup which is either homocyclic or a so-called
Suzuki $2$-group of type $A$. Going further, Moret\'{o} and Navarro
proved in \cite{MN} that if $G$ has at most three irreducible
real-valued characters, then $G$ has a cyclic Sylow $2$-subgroup or
a normal Sylow $2$-subgroup which is homocyclic, quaternion of
order $8$, or an iterated central extension of a Suzuki $2$-group
whose center is an elementary abelian $2$-group. In particular, the
groups with at most three irreducible real-valued characters must be
solvable. Indeed, it was even proved in \cite{NST} that a finite
group with at most three degrees of irreducible real-valued
characters must be solvable.

In a more recent paper \cite{Tongviet}, the third author studied
groups with four real-valued irreducible characters. Among other
results, he proved that a nonsolvable group with exactly four
real-valued irreducible characters must be the direct product of
$\SL_3(2)$ and an odd-order group. Classifying finite groups with exactly five real-valued irreducible
characters seems to be a difficult problem. In the next result, which we prove in Section \ref{Five}, we
control the nonsolvable part of those groups.  We write $\Sol(G)$
to denote the solvable radical of $G$, i.e. the largest solvable
normal subgroup of $G$.

\begin{thmalph}\label{fivervs}
Suppose that a finite group $G$ has at most five real-valued
irreducible characters. Then $G/\Sol(G)$ is isomorphic to the
trivial group, $\SL_3(2)$, $\Al_5$, $\PSL_2(8)\cdot 3$ or
${}^2{\rm{B}}_2(8)\cdot 3$.
\end{thmalph}

Theorem \ref{fivervs} and the aforementioned results
suggest that the nonsolvable part of a finite group perhaps is
bounded in terms of the number of real-valued irreducible characters
of the group. We obtain the following result, proved in Section \ref{Bound}, which
provides a partial answer to Iwasaki's problem.

\begin{thmalph}\label{theorem1} There exists an integer-valued function $f$ on positive integers
such that if G is a finite group with at most $k$ real-valued
irreducible characters, then $|G/\Sol(G)| \leq f(k)$.
\end{thmalph}

Our arguments would allow us to find explicit bounds in Theorem
\ref{theorem1}, but these bounds perhaps are far from best possible.
Therefore, for the sake of simplicity, we have not tried to find the
best bounding function.

Our proof of Theorem \ref{theorem1} uses the classification of finite simple
groups and the following statement for simple groups, proved in Section \ref{Infinity}, which may be of independent interest.

\begin{thmalph}\label{thm:mainsimple}
For a finite nonabelian simple group $S$ and $S \trianglelefteq G
\leq \Aut(S)$, let $k_\bbR(G|S)$ denote the number of real-valued
irreducible characters of $G$ whose kernels do not contain $S$. Then $k_\bbR(G|S) \rightarrow \infty$ as $|S| \rightarrow \infty$.
\end{thmalph}

There is no rational-valued analogue of Theorem \ref{theorem1}, as
shown by the simple groups $\PSL_2(3^{2k+1})$ with $k\geq 1$. We
also note that, by Brauer's permutation lemma, the number of
real-valued irreducible characters and that of conjugacy classes of
real elements in a finite group are always the same.

\section{Nonsolvable groups with five real-valued irreducible characters} \label{Five}

For a finite group $G$, we denote by $\Real(G)$ the set of all real
elements of $G$, $\Elt(G)$ the set of orders of real elements of $G$
and $\Irr_\R(G)$ the set of real-valued irreducible characters of
$G$. Recall that  the generalized Fitting subgroup $\Fit^*(G)$ of
$G$  is the central product of  the layer $\Layer(G)$ of $G$ (the
subgroup of $G$ generated by all quasisimple subnormal subgroups of
$G$) and the Fitting subgroup $\Fit(G)$. Note that if $G$ has a
trivial solvable radical, that is, it has no nontrivial normal
solvable subgroups, then $\Fit^*(G)=\Layer(G)$ is a direct product
of nonabelian simple groups. Moreover, if $\Fit^*(G)$ is a
nonabelian simple group, then $G$ is an almost simple group with
socle $\Fit^*(G)$.

\begin{lem}\label{lem:reduction}
Let $G$ be a finite nonsolvable group with a trivial solvable
radical. If $|\Elt(G)|\leq 5$, then $G$ is an almost simple group.
\end{lem}

\begin{proof}
As the solvable radical of $G$ is trivial, we see that
$\Fit^*(G)=\Layer(G)=\prod_{i=1}^rS_i$ is a direct product of
nonabelian simple groups $S_i$  $(1\leq i\leq r),$ for some integer
$r\ge 1$. It suffices to show that $r=1$.

Suppose by contradiction that $r\ge 2.$ Let $M=S_1\times S_2.$ Since
$\Elt(M)\subseteq \Elt(\Fit^*(G))\subseteq \Elt(G)$, we deduce that
$|\Elt(M)|\leq 5.$ Observe that if $x_i\in \Real(S_i)$ for $i=1,2$,
then $x_1x_2\in\Real(M)$ and thus if $x_1$ and $x_2$ have coprime
orders, then $$o(x_1x_2)=o(x_1)o(x_2)\in\Elt(M).$$

We consider the following cases.

(i): $S_1$ or $S_2$ has no real element of order $4$. Without loss,
assume that $S_1$ has no real element of order $4$. Then by
\cite[Proposition 3.2]{Tongviet},  $S_1$ is isomorphic to one of the
following groups:

$$ \begin{array}{c} \SL_2(2^f) (f\ge 3), \PSU_3(2^f) (f\ge 2), {}^2{\rm B}_2(2^{2f+1}) (f\ge 1); \\
 \PSL_2(q) (5\leq q\equiv 3,5 \:{(\text{mod $8$})}),
{\rm J}_1, {}^2{\rm G}_2(3^{2f+1}) (f\ge 1).
\end{array}$$

By \cite[Theorem 3.1]{NST}, $S_1$ contains real elements $z_1$ and
$z_2$ of order $p_1$ and $p_2$, where $p_1\neq p_2$ are odd primes.
Since $S_2$ has a real element of order $2$, we see that $M$ has
real elements of order $1,2,p_1,p_2,2p_1,2p_2$, which is impossible.

(ii): Both $S_1$ and $S_2$ have real elements of order $4$. By
\cite[Proposition 6.4]{DNT}, $S_1$ contains a real element of order
$p$, where $p$ is an odd prime. Since $S_2$ has real elements of
order $2$ and $4$, $M$ has real elements of order $1,2,4,p,2p$ and
$4p$, which is impossible again.
\end{proof}

Next, we classify all finite nonabelian simple groups $S$ with
$|\Elt(S)|\leq 5.$ Recall that a finite group $G$ is called a
$(C)$-group if the centralizer of every involution of $G$ has a
normal Sylow $2$-subgroup. By \cite[Lemma 2.7]{DGN}, $G$ is a
$(C)$-group if and only if $G$ has no real element of order $2m$
with $m>1$ being odd.

\begin{lem}\label{lem:simple 5 element orders}
Let $S$ be a nonabelian simple group. Then $|\Elt(S)|\leq 5$ if and
only if $S$ is isomorphic to one of the following simple groups:

\begin{enumerate}[$(1)$]
\item $\Al_5\cong\PSL_2(4)$, $\SL_3(2)$, $\PSL_3(3)$, or $\PSU_3(3)$;

\item $\PSL_2(8),$ $\Al_6\cong\PSL_2(9), \PSL_2(11),\PSL_2(27),\PSU_3(4),\PSL_3(4),{}^2{\rm{B}}_2(8)$;

\item $\PSL_2(3^f)$, where $f\ge 7$ is an odd prime, $3^f+1=4r$, $3^f-1=2s$, and $r,s$ are distinct odd primes.
\end{enumerate}
\end{lem}

\begin{proof}
By \cite[Theorem B]{Tongviet}, we have $4\leq |\Elt(S)|\leq 5.$  If
$S\cong \PSL_3(3)$, $\PSU_3(3)$ or $\SL_3(2)$, then $|\Elt(S)|=5$
and these groups appear in part $(1).$ Assume that $S$ is not isomorphic
to one of these groups.  By \cite[Theorem 3.1]{NST}, $\Elt(S)$
contains at least two distinct odd primes, say $p_1$ and $p_2$.

Assume first that $|\Elt(S)|= 4.$ Then $\Elt(S)=\{1,2,p_1,p_2\}$. It follows that $S$ is a $(C)$-group and thus $S\cong\Al_5$ by
\cite[Theorem 3.1]{Tongviet}. 

Assume next that $|\Elt(S)|= 5.$ Assume that $4\in\Elt(S)$.
Then $\Elt(S)=\{1,2,4,p_1,p_2\}$. In particular, $S$ has no real
element of order $2m$ with $m>1$ odd. By
\cite[Theorem 1]{Suzuki1}, $S$ is isomorphic to $\PSL_2(p)$ where $p$ is a
Fermat or a Mersenne prime; $\Al_6$; or
$\PSL_2(q),{}^2\textrm{B}_2(q),\PSU_3(q)$, or $\PSL_3(q)$ where $q>2$
is a power of $2$. Assume that $4\not\in\Elt(S)$. By
\cite[Proposition 3.2]{Tongviet} $S$ is isomorphic to one of the
following simple groups:
$$
 \begin{array}{c} \SL_2(2^f) (f\ge 3), \PSU_3(2^f) (f\ge 2), {}^2{\rm B}_2(2^{2f+1}) (f\ge 1); \\
 \PSL_2(q) (5\leq q\equiv 3,5 {(\text{mod $8$})}),
{\rm J}_1, {}^2{\rm G}_2(3^{2f+1}) (f\ge 1).
\end{array}
$$

We can check that $\Al_6$ has five distinct real element orders but
$\textrm{J}_1$ has more than five distinct real element order. So we
may assume that $S$ is not one of these two groups. We now consider
the following cases.

\smallskip

\textbf{Case 1}: $S\cong \SL_2(2^f),f\ge 3.$ If $3\leq f\leq 6$,
then we can check that only $\SL_2(8)$ has exactly five distinct
real element orders. So, assume $f>6.$ Since $\SL_2(2^f)$ contains
real elements $x$ and $y$ of order $2^f-1$ and $2^f+1$,
respectively, together with real elements of order $1$ and $2$, we
deduce that one of the numbers $2^f\pm 1$ is an odd prime and the
other is a square of an odd prime. Since $f\ge 6$, we can check that
this cannot occur.

\smallskip

\textbf{Case 2}: $S\cong \PSL_2(q)$, $q=p^f$, where $f\ge 1$ and
$p>2$ is a prime. Using \cite{GAP}, if $q\leq 37$, then $q\in
\{9,11,27\}$. Assume that $q\equiv \epsilon$ mod $4$, $\epsilon=\pm
1. $ In this case, $S$ has real elements of order order
$(q-\epsilon)/2$ and $(q+\epsilon)/2$, respectively. Note that
$(q-\epsilon)/2$ is even.

Assume that $(q-\epsilon)/2=2^{a}$ for some integer $a\ge 1.$ Since
$q>37$, $2^a\ge 16$ and thus $S$ has real elements of order
$1,2,4,8,16$ and $(q+\epsilon)/2$, which is a contradiction. Thus
$(q-\epsilon)/2$ is divisible by $2r$ for some odd prime $r$. Let
$s$ be a prime divisor of $(q+\epsilon)/2$. Then
$\{1,2,r,s,2r\}\subseteq \Elt(S)$ and since $|\Elt(S)|=5$,
$(q-\epsilon)/2=2r$ and $(q+\epsilon)/2=s$, where $r,s$ are distinct
odd primes. If $p>3$, then since $3\mid q^2-1$, we must have $r=3$ or
$s=3$, which is not the case as $q>37$.  Therefore $p=3$ and
$q=3^f>37$ so $f\ge 4.$ If $f$ is even, then $q\equiv 1$ mod $8$ and
thus $(q-1)/2$ is divisible by $4$ which is impossible. Thus
$f\ge 5$ is odd and so $\epsilon=-1$. Hence $(3^f+1)/2=2r$ and
$(3^f-1)/2=s.$ The latter equation forces $f$ to be a prime. This is
part $(3)$ of the lemma. Direct calculation shows that $f\ge 7.$

\smallskip

\textbf{Case 3}: $S\cong \PSU_3(2^f),f\ge 2.$ In this case, $S$ has
a subgroup $T\cong \SL_2(2^f)$.  From Case 1, we must have $f=2$ or
$3$. However $|\Elt(\PSU_3(8))|=6$ so $S\cong \PSU_3(4)$.

\smallskip

\textbf{Case 4}: $S\cong{}^2\textrm{B}_2(2^{2f+1}),f\ge 1.$ If
$f=1$, then we can check that $S$ satisfies the hypothesis of the
lemma. Assume $f\ge 2.$ By \cite[Theorem 9 and Proposition 16]{Suzuki2}, $S$ has three nontrivial real elements of odd
distinct orders, which are $2^{2f+1}-1$ and $2^{2f+1}\pm2^{f+1}+1$.
As $|\Elt(S)|=5$, all of these numbers must be primes. Hence
$$4^{2f+1}+1=(2^{2f+1}+2^{f+1}+1)(2^{2f+1}-2^{f+1}+1)$$ is a product
of two distinct primes. Since $5$ divides  $4^{2f+1}+1$, we deduce
that $$2^{2f+1}-2^{f+1}+1=5$$ which is impossible as $f\ge 2.$

\smallskip

\textbf{Case 5}: $S\cong \PSL_3(2^f),f\ge 2.$ As $S$ contains a
subgroup isomorphic to $\SL_2(2^f)$, we deduce that $f=2$ or $3$.
Using \cite{GAP}, only $\PSL_3(4)$ satisfies the hypothesis of the
lemma.

\smallskip

\textbf{Case 6}: $S\cong {}^2{\rm G}_2(3^{2f+1}),f\ge 1.$ In this
case, $S$ contains subgroups isomorphic to $\PSL_2(3^{2f+1})$ and
$\PSL_2(8)$. Since $|\Elt(\PSL_2(8))|=5$, we deduce that
$$\Elt(\PSL_2(3^{2f+1}))\subseteq \Elt(\PSL_2(8))=\{1,2,3,7,9\}.$$
Thus $(3^{2f+1}+1)/2\leq 9$ as $\PSL_2(3^{2f+1})$ has a real element
of order $(3^{2f+1}+1)/2$. Therefore,  $3^{2f+1}\leq 17$ which is
impossible as $f\ge 1.$

Conversely, if $S$ is one of the simple groups in $(1)$-$(3)$, then we
can check that $S$ has at most $5$ distinct real element orders.
\end{proof}

\begin{lem}\label{lem:5 real characters-almost simple}
Let $G$ be an almost simple group with a nonabelian simple socle
$S$. Then
\begin{enumerate}[$(1)$]
\item $G$ has exactly four real-valued irreducible characters if and only if $G\cong\SL_3(2)$.

\item $G$ has exactly five real-valued irreducible characters if and only if
$G$ is isomorphic to $\Al_5,\PSL_2(8)\cdot 3$ or ${}^2{\rm B}_2(8)\cdot 3.$
\end{enumerate}
\end{lem}

\begin{proof} Part (1) follows from \cite[Theorem 3.3]{Tongviet}.
Assume that $G$ is an almost simple group with a nonabelian simple
socle $S$ and that $G$ has exactly five real-valued irreducible
characters. By Brauer's Lemma on character tables, $G$ has exactly
five conjugacy classes of real elements and thus $|\Elt(G)|\leq 5.$
Hence $|\Elt(S)|\leq 5$ as $\Elt(S)\subseteq \Elt(G)$. Therefore $S$
is one of the simple groups appear in the conclusion of Lemma
\ref{lem:simple 5 element orders}. If $S$ is one of the groups in
$(1)-(2)$ of Lemma \ref{lem:simple 5 element orders}, then by using \cite{GAP}, $G$ is isomorphic to
$\Al_5,\PSL_2(8)\cdot 3$ or ${}^2{\rm B}_2(8)\cdot 3$.

Now assume that $S\cong\PSL_2(q)$ with $q=3^{f}$, where $f\ge 7$ is
a prime, $3^f+1=4r$ and $3^f-1=2s$, where $r,s$ are distinct odd
primes. Let $x\in S$ be a real element of order $s$. Then $\langle
x\rangle$ is a Sylow $s$-subgroup of $S$ and its normalizer  in $S$
is a dihedral group of order $2s$. It follows that $S$ has $(s-1)/2$
conjugacy classes of real elements of order $s$. Since
$|\Out(S)|=2f$, we see that $G$ has at least $(s-1)/(4f)$ conjugacy
classes of real elements of order $s$. Since $G$ already has $4$
conjugacy classes of real elements of orders $1,2,r,2r,$ we deduce
that $G$ must have exactly one conjugacy class of real element of
order $s$. Since $(s-1)/(4f)=(3^f-3)/(8f)$ and $f\ge 7$ is a prime,
we can check that $(3^f-3)/(8f)>1$. Thus this case cannot occur.
\end{proof}

%--------------------------------------------

The next theorem proves Theorem \ref{fivervs}, and provides additional information.

\begin{thm}\label{th: 5 real characters}
Let $G$ be a finite group. Assume that $G$ has at most 
five real-valued irreducible characters. Then $G$ is either solvable or $G/\Sol(G)\cong
\SL_3(2),\Al_5,\PSL_2(8)\cdot 3,$ ${}^2{\rm{B}}_2(8)\cdot 3$. Moreover,
if $|\Irr_\RR(G/\Sol(G))|=5,$ then one of the following holds.
\begin{enumerate}[$(1)$]
\item $G\cong \Al_5\times K$, where $K$ is of odd order.
\item $G\cong (L\times K)\cdot 3$, where $L\cong\PSL_2(8)$ or ${}^2{\rm{B}}_2(8)$ and $K$ is of odd order.
\end{enumerate}
\end{thm}

\begin{proof}
We may assume that $G$ is nonsolvable and 
$|\Irr_\R(G)|\leq 5.$  Then $|\Irr_\R(G/\Sol(G))|\leq 5$ and thus  $|\Elt(G/\Sol(G))|\leq 5$. By Lemma
\ref{lem:reduction},  $G/\Sol(G)$ is an almost simple group.
 It follows from \cite[Theorem B]{Tongviet} that $G/\Sol(G)$ has at least
four real-valued irreducible characters; hence $4\leq
|\Irr_\R(G/\Sol(G))|\leq 5.$ Now Lemma \ref{lem:5 real
characters-almost simple} yields the first part of the theorem.

Assume that $|\Irr_\R(G/\Sol(G))|=5.$ By Lemma \ref{lem:5 real
characters-almost simple}, $G/\Sol(G)\cong \Al_5,\PSL_2(8)\cdot 3$
or ${}^2\textrm{B}_2(8)\cdot 3.$ In all cases, $G/\Sol(G)$ has $3$
conjugacy classes of nontrivial real elements of odd orders. As real
elements of odd order of $G/\Sol(G)$ lift to real elements of odd
order of $G$ by \cite[Lemma 2.2]{GNT}, $G$ has three conjugacy
classes of nontrivial real elements of odd orders. It follows that
$\Sol(G)$ has no nontrivial real element of odd order and thus
$\Sol(G)$ has a normal Sylow $2$-subgroup by \cite[Proposition
6.4]{DNT}. Moreover, as $|\Irr_\R(G)|\leq 5,$ the above argument shows that $G$ has no real element of order $2m$ with $m>1$ being odd, so $G$ is a
$(C)$-group and has
no real element of order $4$. By \cite[Theorem 2.3]{Tongviet},
\[L=\OO^{2'}(G)\cong \SL_2(2^f) (f\ge 2) \text{ or }
{}^2\textrm{B}_2(2^{2f+1}) (f\ge 1).\] It follows that $L\cong
\Al_5,\PSL_2(8) \text{ or } {}^2\textrm{B}_2(8)$ as these are the
only possible nonabelian composition factors of $G$.

Let $K:=\Centralizer_G(L)$. Then $K\cap L=1$ and $K\times L\unlhd
G.$ Since $G$ is a $(C)$-group, we deduce that $|K|$ is odd. Now
$G/K$ is isomorphic to a subgroup of $\Aut(L)$ and $|\Irr_\R(G/K)|\leq
5,$ we conclude that either $G=K\times \Al_5$ or $G=(L\times
K)\cdot 3$, where $L\cong \PSL_2(8)$ or ${}^2\textrm{B}_2(8)$, as
claimed.
\end{proof}

\section{Real-valued characters of almost simple groups} \label{Infinity}
In this section we prove Theorem \ref{thm:mainsimple}.  We begin with the following observation:

\begin{lemma}\label{lem:roughbound}
Keep the notation as in Theorem \ref{thm:mainsimple}. Then
$k_\bbR(G|S) = k_\bbR(G) - k_\bbR(G/S)$,  $k_\bbR(G) \geq k_\bbR(S) / |\out(S)|,$ and $ k_\bbR(G/S) \leq |\out(S)|.$
In particular, we have
\[k_\bbR(G|S) \geq k_\bbR(S)/|\out(S)| - |\out(S)|.\]
\end{lemma}

For $S$ a finite nonabelian simple group, we will write $$\mathfrak{K}(S):= k_\bbR(S)/|\out(S)| - |\out(S)|.$$ Hence, to prove Theorem \ref{thm:mainsimple}, it suffices to show that $\mathfrak{K}(S)\rightarrow\infty$ as $|S|\rightarrow\infty$.  We remark that due to the nature of the statement of Theorem \ref{thm:mainsimple}, we may disregard a finite number of simple groups.  Recall that we may view $k_\bbR(S)$ as either the number of real-valued irreducible characters or the number of real conjugacy classes of $S$.

\subsection{Initial Considerations}

Throughout, when $q$ is a power of a prime $p$, we will write $\nu(q)$ for the positive integer such that $q=p^{\nu(q)}$.

\begin{lemma}\label{lem:initial1}
Let $S$ be a simple group isomorphic to the alternating group $\Al_n$ for $n\geq 5$, or a simple group of Lie type   ${}^2{\rm B}_2(q)$,  ${}^2{\rm G}_2(q)$, ${\rm G}_2(q)$, ${}^3{\rm D}_4(q)$, ${\rm D}_4(q)$, ${}^2{\rm F}_4(q),$ ${\rm F}_4(q)$, or ${}^2{\rm D}_{2n}(q)$ with $n\geq 2$ for $q$ a power of a prime.  Or assume $q\not\equiv 3\pmod 4$ and that $S$ is a simple group of Lie type ${\rm B}_n(q)$ with $n\geq 3$, ${\rm C}_n(q)$ with $n\geq 1$, or ${\rm D}_{2n}(q)$ with $n\geq 3$. Then $\mathfrak{K}(S) \rightarrow\infty$ as $|S| \rightarrow\infty$.
\end{lemma}
\begin{proof}
Note that every conjugacy class of the symmetric group $\Sy_n$ is real, and that a class in $\Sy_n$ yields exactly one class in $\Al_n$ if and only if the cycle type of the elements in the classes contain an even cycle or two cycles of the same length.  Since the number of cycle types of this form is increasing with $n$ and $|\out(\Al_n)|=2$ for $n\geq 7$, we see that the statement holds if $S = \Al_n$.

If $S$ is ${\rm G}_2(q)$, ${}^2{\rm G}_2(q)$, ${}^{2}{\rm B}_2(q)$, or ${}^{2}{\rm F}_4(q)$, then observing the generic character tables available in CHEVIE \cite{chevie}, we see that all except two, six, two, or twelve, respectively, of the characters are real-valued.  If $S$ is ${\rm F}_4(q)$, then all except four of the characters are real-valued, using \cite[Theorem 4.1]{TV19}.  If $S$ is one of the remaining simple groups of Lie type listed, then \cite[Theorem 1.2]{tiepzalesski05} yields that every element of $S$ is real.

Further, $\out(S)\leq 24\nu(q)$ where $q=p^{\nu(q)}$ for a prime $p$.  Hence since the number of classes in $S$ can be written as a polynomial in $p$ whose exponents are in terms of $n$ and $\nu(q)$, the statement also holds in these cases.
\end{proof}

\subsection{Fixed Parameters and Classical Groups}

\begin{proposition}\label{prop:fixrank}
Let $\mathbb{S}$ be a family of simple groups of Lie type with the same type and rank.  That is, there is some generic reductive group $\mathbb{G}$ as in \cite[Section 2.1]{brouemalle} of simply connected type such that for each $S\in\mathbb{S}$, $S$ is of the form $G/Z(G)$ where $G=\mathbb{G}(q)$ for some prime power $q$.
Then for $S\in\mathbb{S}$, we have $\mathfrak{K}(S) \rightarrow\infty$ as $q \rightarrow\infty$.  
\end{proposition}
\begin{proof}

Let $q$ be a power of a prime $p$ and let $G=\mathbb{G}(q)$ be the fixed points $\bg{G}^F$ of a simple simply connected algebraic group $\bg{G}$ over $\overline{\bbF}_q$ under a Frobenius morphism $F$.   Let $\bg{T}$ be a rational maximal torus of $\bg{G}$ and let $\Phi$ and $\Delta$ be a root system and set of simple roots, respectively, for $\bg{G}$ with respect to $\bg{T}$. Let $|\Delta|=n$.  We use the notation as in \cite{gls} for the Chevalley generators.   In particular, note that $\bg{T}$ is generated by $h_{\alpha}(t)$ for $t\in\overline{\bbF}_q^\times$ and  $\alpha\in\Phi$, and $\bN_{\bg{G}}(\bg{T})$ is generated by $\bg{T}$ and the $n_\alpha(1)$ for $\alpha\in\Phi$.

Note that we may assume that $\mathbb{S}$ is not one of the families considered in Lemma \ref{lem:initial1}.   Hence if $F$ is twisted, we may assume that $F=\tau F_q$ where $|\tau|=2$ and either $\bg{G}$ is type ${\rm A}_n$ or $n>4$. Here $F_q$ denotes the standard Frobenius morphism induced from the map $x\mapsto x^q$ on $\overline{\bbF}_q$, and $\tau$ denotes a graph automorphism of $\bg{G}$.

First assume $F$ is not twisted, so $\tau=1$.   Fix some $\alpha\in \Delta$.  Then for $t\in\bbF_{q}^\times$, we know $s:=h_\alpha(t)$ is real in $G$, with reversing element $n_\alpha(1)$.  
Further, $s$ lies in the maximally split torus $T:=\bg{T}^F$.  By \cite[Cor. 0.12]{dignemichel}, we know that $\bN_G(\bg{T})$ controls fusion in $\bg{T}$, so if $s=h_\alpha(t)$ and $s':=h_{\alpha}(t')$ for $t, t'\in \bbF_q^\times$ are conjugate, then they are conjugate in $\bN_{\bg{G}}(\bg{T})$.  In particular, this means there is some product $x:=\prod_{\beta\in J} n_\beta(1)$ with $J\subseteq \Phi$ that conjugates $s$ to $s'$.

Then by the properties of the Chevalley generators from \cite[1.12.1]{gls}, we see this is impossible unless $t'= \pm t^{\pm1}$. Indeed, $h_\alpha(t)^{n_\beta(1)}=h_{r_\beta(\alpha)}(\pm t)$, so we may write $h_\alpha(t)^x=h_{r(\alpha)}(\pm t)$ where $r:=\prod_{\beta \in J} r_\beta$ is the corresponding composition of reflections.  Then if $h_\alpha(t)^x=h_\alpha(t')$, we have $h_{r(\alpha)}(\pm t)=h_\alpha(t')$.  But if $\widecheck{r(\alpha)}=\sum_{i=1}^n c_i \widecheck{\alpha_i}$, then $h_{r(\alpha)}(\pm t)=\prod_{i=1}^n h_{\alpha_i}(\pm t^{c_i})$. Here $\Delta:=\{\alpha_1,\ldots,\alpha_n\}$ and for any $\beta\in\Phi$ we write $\widecheck{\beta}=2\beta/(\beta,\beta)$. Further, since $\bg{G}$ is simply connected, there is an isomorphism $(\overline{\bbF}_q^\times)^n\rightarrow \bg{T}$ given by $(t_1,\ldots t_n)\mapsto h_{\alpha_1}(t_1)\cdots h_{\alpha_n}(t_n)$ (see \cite[1.12.5]{gls}).  Then this yields $c_i=0$ for $\alpha_i\neq \alpha$, and hence $\widecheck{r(\alpha)}=c\widecheck{\alpha}$ for some integer $c$ and $t'=\pm t^c$.  Then since $(r(\alpha), r(\alpha))=(\alpha, \alpha)$, we have $r(\alpha)=c\alpha$ and $c=\pm1$.

Now by \cite[1.12.6]{gls}, we see $s\not\in Z(G)$ for $\delta\neq \pm1$.  This yields that $k_\bbR(G)\geq (q-3)/2$, and since (except for a finite number of exceptions) $|Z(G)|\leq n+1$ and $|\out(S)|\leq 2(n+1)\nu(q)$, we see
$\mathfrak{K}(S)\geq \frac{(q-3)}{8(n+1)^2\nu(q)}-2(n+1)\nu(q)\rightarrow\infty$ as $q\rightarrow\infty$.

If $\tau\neq 1$, we may argue similarly, taking $\alpha\in\Delta$ to be fixed by $\tau$, unless $\bg{G}$ is type ${\rm A}_n$ with $n$ even.  In the latter case, we may instead take $h_{\alpha_1}(t)h_{\alpha_n}(t^q)$ with $t\in \bbF_{q^2}^\times$. Then in each case, the element being considered lies in $T=\bg{T}^F$ (see \cite[2.4.7]{gls}) and similar arguments to above show that we still have $\mathfrak{K}(S) \rightarrow \infty$ as $n\rightarrow\infty$.
\end{proof}

\begin{corollary}\label{cor:fixrank}
Let $S(q)$ be a simple group ${\rm E}_6(q)$, ${}^2{\rm E}_6(q)$, ${\rm E}_7(q)$, ${\rm E}_8(q)$, ${\rm A}_n(q)$, or ${}^2{\rm A}_n(q)$, with $n$ a fixed positive integer.  Then $\mathfrak{K}(S(q)) \rightarrow\infty$ as $q\rightarrow\infty$.
\end{corollary}

\begin{lemma}\label{lem:fixq}
Let $q$ be a fixed power of a prime and let $S_n$ be a simple group of Lie type of classical type: ${\rm A}_n(q)$, ${}^2{\rm A}_n(q)$, ${\rm B}_n(q)$, ${\rm C}_n(q)$, ${\rm D}_n(q)$, or ${}^2{\rm D}_n(q)$.    Then $\mathfrak{K}(S_n) \rightarrow\infty$ as $n\rightarrow\infty$.
\end{lemma}
\begin{proof}
Consider the unipotent characters of $S_n$.  By \cite{lusztig}, these characters are real-valued.  For $S_n$ of the form ${\rm A}_n(q)$ or ${}^2{\rm A}_n(q)$, these characters are indexed by partitions of $n+1$ (see \cite[13.8]{carter}). Since the number of these behaves asymptotically like $\frac{\mathrm{exp}(\pi\sqrt{2(n+1)/3})}{4(n+1)\sqrt{3}}$ as $n\rightarrow\infty$ (see \cite[(5.1.2)]{andrews}), and $|\out(S)|\leq 2(q+1)\nu(q)$, we have $\mathfrak{K}(S_n(q))\rightarrow\infty$ as $n\rightarrow\infty$.

For $S_n(q)$ of the form $\mathrm{B}_n(q)$, ${\rm C}_n(q)$, ${\rm D}_n(q)$, or ${}^2{\rm D}_n(q)$, the unipotent characters are indexed by symbols as in \cite[13.8]{carter}, the number of which is at least the number of partitions of $n-1$.   Then since $|\out(S_n)|\leq 8\nu(q)$ for $n\geq 5$, we have $\mathfrak{K}(S_n(q))\rightarrow\infty$ as $n\rightarrow\infty$ again in this case.
\end{proof}

\begin{proposition}\label{prop:classicals}
 Let $S_n(q):=\Omega_{2n+1}(q)$, $\PSp_{2n}(q)$, or $\mathrm{P}\Omega_{2n}^\pm(q)$, with $n\geq 5$.  Then $\mathfrak{K}(S_n(q)) \rightarrow\infty$ as $nq \rightarrow\infty$.
\end{proposition}
\begin{proof}
Write $S=S_n(q)$ as $G/Z(G)$, where $G=\bg{G}^F$ is the set of fixed points of a simple, simply connected algebraic group $\bg{G}$ over $\overline{\bbF}_q$ under a Frobenius morphism $F$. Notice that $|Z(G)|\leq 4$.

  Let $\bg{T}$ be a rational maximal torus of $\bg{G}$ and let $\Phi$ and $\Delta$ be a root system and set of simple roots, respectively, for $\bg{G}$ with respect to $\bg{T}$. Here we have $|\Delta|=n$ and $\Phi$ is of type ${\rm B}_n$, ${\rm C}_n$, or ${\rm D}_n$.  We use the notation as in \cite[1.12.1]{gls} for the Chevalley generators.   Recall that $\bg{T}$ is generated by $h_{\alpha}(t)$ for $t\in\overline{\bbF}_q^\times$ and $\alpha\in\Phi$ and that $\bN_{\bg{G}}(\bg{T})$ is generated by $\bg{T}$ and the $n_\alpha(1)$ for $\alpha\in\Phi$.

We will use the standard model as in \cite[Remark 1.8.8]{gls} for the members of $\Delta$.  Namely, let $\{e_1,\ldots e_n\}$ be an orthonormal basis for the $n$-dimensional Euclidean space and let $\Delta=\{\alpha_1,\ldots,\alpha_n\}$.  Note that for $1\leq i\leq n-1$, we have $\alpha_i:=e_i-e_{i+1}$. Further, since $\bg{G}$ is simply connected, there is an isomorphism $(\overline{\bbF}_q^\times)^n\rightarrow \bg{T}$ given by $(t_1,\ldots t_n)\mapsto h_{\alpha_1}(t_1)\cdots h_{\alpha_n}(t_n)$ (see \cite[1.12.5]{gls}).

Using Lemmas \ref{lem:initial1} and \ref{lem:fixq}, we may suppose $q\geq 3$.  For each $1\neq\delta\in\bbF_q^\times$, we let $s_0(\delta):=h_{\alpha_1}(\delta)$, $s_1(\delta):= h_{\alpha_1}(\delta)h_{\alpha_3}(\delta)$, and in general for $0\leq m\leq \lceil \frac{n-4}{2}\rceil$, let $s_m(\delta):=\prod_{k=0}^m h_{\alpha_{2k+1}}(\delta)$.  Our choices of $m$ ensure that $s_m(\delta)\in\bg{G}$ is fixed by $F$, since in the case of ${}^2{\rm D}_n(q)=\mathrm{P}\Omega_{2n}^-(q)$, $s_m(\delta)$ is fixed by the graph automorphism and Frobenius $F_q$.  Hence $s_m(\delta)\in G=\bg{G}^F$.  Further, since $(\alpha_i, \alpha_j)=0$ for $|j-i|>1$, we see that each $s_m$ is real in $G$, with reversing element $\prod_{k=0}^m n_{\alpha_{2k+1}}(-1)$.

Recall that $\bN_{\bg{G}}(\bg{T})$ controls fusion in $\bg{T}$.  Hence if $s_m(\delta)$ and $s_{m'}(\delta')$ are conjugate in $G$, then there is some $w\in W:=\bN_{\bg{G}}(\bg{T})/\bg{T}$ such that $s_m(\delta)^w=s_{m'}(\delta')$.  But $W\leq C_2\wr \Sy_n$ (with $C_2$ the group of order 2) where the generators of the base subgroup $C_2^n$ act via $e_i\mapsto -e_i$ and the copy of $\Sy_n$ permutes the $e_i$'s.  Then by the properties of the Chevalley generators from \cite[1.12.1]{gls}, we see this is impossible unless $m=m'$ and $\delta=\delta'$.

Note that we have $(q-3)/2$ choices for $\delta\neq \pm1$, since $s_m(\delta)$ is conjugate to $s_m(\delta^{-1})$, giving $(q-3)/2+1=(q-1)/2$ elements in this form for a fixed $m$. Further, since $\delta\neq 1$ and $s_m(\delta)$ has no factor $h_{\alpha_n}(\delta)$ nor $h_{\alpha_{n-1}}(\delta)$, we see by \cite[1.12.6]{gls} that $s_m(\delta)\not\in Z(G)$, so we have $k_{\R}(S)>\frac{1}{4} \lceil \frac{n-4}{2}\rceil \left(\frac{q-1}{2}\right)$. Further, $|\out(S)|\leq 8\nu(q)$, so
\[\mathfrak{K}(S_n(q)) > \frac{(n-4)(q-1)}{16\cdot 8\nu(q)}-8\nu(q)\] which tends to $\infty$ as $nq\rightarrow\infty$.
\end{proof}

\subsection{Linear and Unitary Groups}\label{sec:typeA}
We write $\SL_n^\epsilon(q)$ with $\epsilon\in\{\pm1\}$ to denote $\SL_n(q)$ for $\epsilon=1$ and $\SU_n(q)$ for $\epsilon=-1$, and similarly for $\GL_n^\epsilon(q)$ and $\PSL_n^\epsilon(q)$.
Throughout this section, we also write  $\wt{G}=\GL_n^\epsilon(q)$, $G=\SL_n^\epsilon(q)=[\wt{G}, \wt{G}]$, and $S=G/Z(G)=\PSL_n^\epsilon(q)$.  Note that $\wt{G}\cong \wt{G}^\ast$ in this case, where $\wt{G}^\ast$ denotes the dual group, and we make this identification.

If $s$ is a semisimple element of $\wt{G}$, there exists a unique semisimple character $\wt{\chi}_s$ associated to the $\wt{G}$-conjugacy class of $s$, and $\wt{\chi}_{s^{-1}}$ is the complex conjugate character of $\wt{\chi}_s$.  Hence $\wt{\chi}_s$ is real if $s$ is.   If further $s\in G=[\wt{G}, \wt{G}]$, then $\wt{\chi}_s$ is trivial on $Z(\wt{G})$, using \cite[Lemma 4.4]{NavarroTiep13}.  Furthermore, the number of irreducible constituents of $\chi :=\wt{\chi}_s|_G$ is exactly the number of irreducible characters $\theta \in \Irr(\widetilde{G}/G)$ satisfying $\widetilde{\chi}_s\theta = \widetilde{\chi}_s$, and we have $\irr(\widetilde{G}/G) = \{\widetilde{\chi}_z \mid z \in Z(\widetilde{G})\}$.  Also, for such $z \in Z(\widetilde{G})$, if we take the product with $\widetilde{\chi}_z$ of each character in the Lusztig series for $\wt{G}$ indexed by $s$,  we obtain the Lusztig series indexed by $sz$, by \cite[13.30]{dignemichel}. Then $\chi$ is irreducible if and only if $s$ is not $\widetilde{G}$-conjugate to $sz$ for any nontrivial $z\in Z(\widetilde{G})$.  Further, if $s$ and $s'$ are two such elements, an application of Gallagher's theorem \cite[Corollary 6.17]{isaacs} together with the above reasoning yields that if $\widetilde{\chi}_s|_G=\wt{\chi}_{s'}|_G$, then $s$ is conjugate to $s'z$ for some $z \in Z(\widetilde{G})$.

Hence we aim to construct a collection $X$ of real semisimple elements of $G$ such that two elements $s, s'\in X$ satisfy that $s$ and $s'z$ for $z\in Z(\wt{G})$ are $\widetilde{G}$-conjugate if and only if $s'=s$ and $z=1$ and such that $|X|/|\out(S)|-|\out(S)|$ tends to $\infty$ as $nq\rightarrow\infty$.

\begin{proposition}\label{prop:SLn}
 Let $S_n(q):=\PSL_n^\pm(q)$.  Then $\mathfrak{K}(S_n(q))  \rightarrow\infty$ as $nq \rightarrow\infty$.
\end{proposition}
\begin{proof}
By Corollary \ref{cor:fixrank} and Lemma \ref{lem:fixq}, we may assume that $q$ and $n$ are sufficiently large.   Write $\bar{n}:=\lfloor n/4\rfloor$.

Recall that the semisimple elements of $\wt{G}$ are completely determined by their eigenvalues.   Consider a semisimple element $$s=s(\lambda_1,\ldots, \lambda_{\bar{n}}):=\diag(\lambda_1, \lambda_1^{-1}, \lambda_2, \lambda_2^{-1},\ldots,\lambda_{\bar{n}}, \lambda_{\bar{n}}^{-1}, I_{n-2\bar{n}})$$ in $G$, where each $\lambda_i$ is an element of the cyclic subgroup $C_{q-\epsilon}$ of $\bbF_{q^2}^\times$ and not all of the $\lambda_i$ are in $\{\pm1\}$.

 We see by the dimension of $\ker(s-1)$ that $s$ is not conjugate to $sz$ for any $1\neq z=\mu I_n\in Z(\wt{G})$, since otherwise $1=\lambda_i\mu=\lambda_i^{-1}\mu$ for each $i$, implying that $\lambda_i^2=1$ for each $i$, contradicting our assumption that not all $\lambda_i$ are in $\{\pm1\}$.  Similarly, if $s'$ is another semisimple element of this form, defined by $\lambda_i'$ for $1\leq i\leq \bar{n}$, such that $s'$ is conjugate to $sz$ with $z=\mu I_n\in Z(\wt{G})$, then it must be that $\mu=1$ and $s$ is conjugate to $s'$.

Then by considering the elements of the form \[s(\lambda_1, 1,\ldots, 1), s(\lambda_1, \lambda_1, 1, \ldots, 1),..., s(\lambda_1,\ldots,\lambda_1),\] together with those of the form \[s(\lambda_1, \lambda_2, 1,\ldots, 1), s(\lambda_1, \lambda_2, \lambda_2, 1,\ldots, 1), \ldots, s(\lambda_1, \lambda_2, \ldots, \lambda_2)\] and \[s(\lambda_1, \lambda_2,  \lambda_3,1,\ldots, 1), s(\lambda_1, \lambda_2, \lambda_3, \lambda_3, 1,\ldots, 1), \ldots, s(\lambda_1, \lambda_2, \lambda_3 \ldots, \lambda_3)\] with $\lambda_1, \lambda_2, \lambda_3$ and their inverses all distinct, we see
\begin{align*}
k_\bbR(S_n(q))&\geq \bar{n}(q-3)/2 + (\bar{n}-1)(q-3)(q-5)/4 + (\bar{n}-2)(q-3)(q-5)(q-7)/8\\
&> \bar{n}(q-5)^3/8 - (q-5)(q-3)(q-6)/4\\
&> \bar{n}(q-5)^3/8 - 2(q-3)^3/8\\
&= \frac{(\bar{n}-2)(q-5)^3-12(q-5)^2-24(q-5)-16}{8}.\\
\end{align*}
So \begin{align*}
\mathfrak{K}(S_n(q))&\geq \frac{(\bar{n}-2)(q-5)^3-12(q-5)^2-24(q-5)-16}{16(q+1)\nu(q)} - 2(q+1)\nu(q)\\
&=  \frac{(\bar{n}-2)(p^{\nu(q)}-5)^3-12(p^{\nu(q)}-5)^2-24(p^{\nu(q)}-5)-16- 4(p^{\nu(q)}+1)^2\nu(q)^2}{16(p^{\nu(q)}+1)\nu(q)}, \\
\end{align*}
 which tends toward $\infty$ as $nq\rightarrow\infty$.
  \end{proof}

Theorem \ref{thm:mainsimple} now follows by combining Lemmas \ref{lem:initial1} and \ref{lem:fixq} with Propositions \ref{prop:fixrank}, \ref{prop:classicals}, and \ref{prop:SLn}.

\section{Proof of Theorem \ref{theorem1}} \label{Bound}

We start with a well-known observation.

\begin{lemma} \label{ExtendLemma} Let $S$ be a finite nonabelian simple group. Then there exists a
non-principal irreducible character of S that is extendible to a
rational-valued character of $\Aut(S)$.
\end{lemma}

\begin{proof}
For each $n\geq5$, consider the irreducible character of the
symmetric group $\Sy_n$ labeled by the partition $(n-1,1)$. This
character restricts irreducibly to the alternating group $\Al_n$. As
it is well known that every character of $\Sy_n$ is rational-valued,
the lemma is proved for the alternating groups. For the sporadic
simple groups and the Tits group, one can check the statement
directly by using \cite{Atl}. Finally, when $S$ is a simple group of
Lie type, the Steinberg character of S extends to a rational-valued
character of $\Aut(S)$, see \cite{Fei} for instance.
\end{proof}

\begin{proposition} \label{BoundProp} Assume that $N = S_1 \times S_2 \times\cdots\times S_n$, a direct product of
copies of a finite nonabelian simple group $S \cong S_i$, is a
normal subgroup of $G$. Then the number of rational-valued
irreducible characters of $G$ is at least $n$.
\end{proposition}

\begin{proof} Modding out $\bC_G(N)$ if necessary, we may assume that $\bC_G(N) = 1$ so
that $N \unlhd G \leq \Aut(N)$. By Lemma \ref{ExtendLemma}, there exists $\theta
\in \Irr(S)$ that is extendible to a rational-valued character, say
$\lambda$, of $\Aut(S)$. For each $1 \leq j \leq n$, set
\[
\psi_j:=\theta\otimes\cdots\otimes\theta\otimes1_{S_{j+1}}\otimes\cdots\otimes
1_{S_n}\in \Irr(N).
\]
Since $\Aut(N)$ acts transitively on the direct factors $S_i$'s of
$N$, the $\Aut(N)$-orbit of $\psi_j$ consists of characters of the
form $\alpha_1\otimes\alpha_2\otimes\cdots\otimes\alpha_n$ where
$\alpha_i\in\{1_{S_i},\theta\}$ for every $1 \leq i \leq n$ and the
number of times that $\theta$ appears in the tensor product is
precisely equal to $j$. This means that the size of the
$\Aut(N)$-orbit containing $\psi_j$ is $n!/j!(n-j)!$. On the other
hand, we see that $\psi_j$ is invariant under
\[
(\Aut(S) \wr \Sy_j) \times (\Aut(S) \wr \Sy_{n-j}),\] and \[ |\Aut(N) :
(\Aut(S) \wr \Sy_j ) \times (\Aut(S) \wr \Sy_{n-j})| = n!/j!(n-j)!.\] We
therefore deduce that $\Aut(S) \wr \Sy_j \times\Aut(S) \wr
\Sy_{n-j}$ is the inertia subgroup of $\psi_j$ in $\Aut(N)$.

Recall that $\theta$ extends to the rational-valued character
$\lambda\in\Irr(\Aut(S))$. Let $V$ be a $\mathbb{C} \Aut(S)$-module
affording $\lambda$. Then $\Aut(S)^j$ acts naturally on $V^{\otimes
j}$, with the character $\lambda\otimes\cdots\otimes\lambda$, and
$\Sy_j$ permutes the $j$ tensor factors of $V^{\otimes j}$. So
$V^{\otimes j}$ becomes a tensor-induced module for $\Aut(S^j) =
\Aut(S) \wr \Sy_j$. Let $\mu$ be the character afforded by this
module. Then, as $\lambda$ is rational-valued, the formula for the
tensor-induced character (see \cite{GI} for instance) implies that
$\mu$ is also rational-valued. We have seen that $\theta^j$ extends
to the rational-valued character $\mu\in\Irr(\Aut(S^j))$. It follows
that $\psi_j$ extends to a rational-valued character of
$I_{\Aut(N)}(\psi_j)$. In particular, $\psi_j$ extends to a
rational-valued character, say $\nu_j$, of $I_G(\psi_j) = G \cap
I_{\Aut(N)}(\psi_j)$. The Clifford correspondence now produces $n$
different rational-valued irreducible characters, namely $\nu_j^G$,
for $1 \leq j \leq n$, of $G$, and the proposition is proved.
\end{proof}

We are now ready to prove Theorem \ref{theorem1}.

\begin{proof}[Proof of Theorem \ref{theorem1}]
Since $k_\RR(G/\Sol(G)) \leq k_\RR(G)$ and $\Sol(G/\Sol(G))$ is
trivial, we may assume with no loss that $\Sol(G)$ is trivial. The
generalized Fitting subgroup of $G$, denoted by $\bF^*(G)$, is then
the direct product of the minimal normal subgroups of $G$, each of
which is a product of copies of a nonabelian simple group. Therefore
$\bC_G(\bF^*(G)) = 1$ and $G \leq \Aut(\bF^*(G))$.

Let $S$ be a simple direct factor of $\bF^*(G)$ and assume that the
number of times that $S$ appears in $\bF^*(G)$ is $n$. By
Proposition \ref{BoundProp}, we know that $n$ is bounded by $k$. It remains to
prove that $|S|$ is bounded in terms of $k$. Notice that if $|S|$ is
bounded in terms of $k$, then the number of choices for $S$
appearing in $\bF^*(G)$ is bounded, and therefore $\bF^*(G)$ is
bounded, which in turn implies that $|G|$ is bounded in terms of
$k$.

Let $N := S_1 \times S_2 \times\cdots \times S_n$ where each $S_i$
is isomorphic to $S$.

We have
$\bC_{\bN_G(S_1)/\bC_G(S_1)}(N\bC_G(S_1)/\bC_G(S_1)) = 1$, and hence
\[S_1 \cong N\bC_G(S_1)/\bC_G(S_1)\unlhd \bN_G(S_1)/\bC_G(S_1) \leq \Aut(S_1).
\] Assume, to the
contrary, that $|S|=|S_1|$ can be arbitrarily large while $k$ is
fixed. Using Theorem \ref{thm:mainsimple}, we then can choose $S_1$ so that
$\bN_G(S_1)/\bC_G(S_1)$ has at least $k^2 + 1$ real-valued
irreducible characters whose kernels do not contain $S_1$. Let
$\lambda$ be one of these characters.

Let $\theta$ be an irreducible constituent of
$\lambda\hspace{-3pt}\downarrow_{S_1}$, and set $\psi:=\theta\otimes
1_{S_2}\otimes\cdots\otimes 1_{S_n}$. Since $S_1\nsubseteq
\Ker(\lambda)$, we see that $\theta$ is nontrivial, and hence the
inertia subgroup $I_G(\psi)$ is contained in $\bN_G(S_1)$. The
Clifford correspondence then implies that, as $\lambda$ (considered
as a character of $\bN_G(S_1)$) lies over $\psi$, $\lambda^G$ is an
irreducible character of $G$. Moreover, $\lambda^G$ is real-valued
since $\lambda$ is.

We have shown that, for each $\lambda$ among $k^2 +1$ real-valued
irreducible characters of $\bN_G(S_1)$ whose kernels do not contain
$S_1$, there corresponds the real-valued irreducible character
$\lambda^G$ of $G$. On the other hand, as \[|G : \bN_G(S_1)| \leq
|\Aut(\bF^*(G)) : \bN_{\Aut(\bF^*(G))}(S_1)| = n \leq k,\] each
real-valued irreducible character of $G$ lies above at most $k$
irreducible characters of $\bN_G(S_1)$. We therefore deduce that $G$
has at least $k + 1$ real-valued irreducible characters, and this is
contradiction. Thus we conclude that $|S|$ is bounded in terms of
$k$, as desired.
\end{proof}

%--------------------------------------------

\end{document}